\documentclass[reqno]{amsart}
\usepackage{amsmath, amssymb, amsfonts}
\newtheorem{theorem}{Theorem}

\newtheorem{corollary}{Corollary}
\newtheorem{lemma}{Lemma}
\newtheorem{proposition}{Proposition}
\theoremstyle{remark}
\newtheorem{remark}{Remark}
\theoremstyle{definition}
\newtheorem{define}{Definition}
\newtheorem{example}{Example}
\newtheorem*{conjecture}{Conjecture}

\newtheorem*{Acknowlegement}{Acknowlegement}

\begin{document}

\title[A note on uniqueness boundary]{A note on uniqueness boundary of holomorphic mappings}

\author{Ninh Van Thu and Nguyen Ngoc Khanh}
\begin{abstract}
In this paper, we prove Huang et al.'s conjecture stated that if $f$ is a holomorphic function on $\Delta^+:=\{z\in \mathbb C \colon |z|<1,~\mathrm{Im}(z)>0\}$ with $\mathcal{C}^\infty$-smooth extension up to $(-1,1)$ such that $f$ maps $(-1,1)$ into a cone $\Gamma_C:=\{z\in \mathbb C\colon |\mathrm{Im} (z)| \leq C|\mathrm{Re} (z)|\}$, for some positive number $C$, and $f$ vanishes to infinite order at $0$, then $f$ vanishes identically. In addition, some regularity properties of the Riemann mapping functions on the boundary and an example concerning Huang et al.'s conjecture are also given.
\end{abstract}

\thanks{This research is funded by the Vietnam National University, Hanoi (VNU) under project number QG.16.07.}
\address{Department of Mathematics, Vietnam National University at Hanoi, 334 Nguyen Trai, Thanh Xuan, Hanoi, Vietnam}
\email{thunv@vnu.edu.vn, khanh.mimhus@gmail.com}
\maketitle


\section{Introduction}
Let $\Omega$ be a domain in $\mathbb R^n$ with $a\in \partial\Omega$. A continuous function $f: \overline{\Omega}\to \mathbb C$ vanishes to infinite order at $a$ if, for every $N\in \mathbb N$,
$$
\lim_{\Omega\ni x\to a}\frac{f(x)}{|x-a|^N}=0. 
$$

In 1991, M. Lakner \cite{La91} proved the following result.  
\begin{theorem}[\cite{La91}] Suppose that $f$ is a function on the upper half disc $\Delta^+:=\{z\in \mathbb C\colon |z|<1,~\mathrm{Im}(z)>0\}$ that extends continuously to the diameter $(-1,1)$, such that the extension maps $(-1,1)$ to a cone $\Gamma_C:=\{z\in \mathbb C\colon |\mathrm{Im} (z)| \leq C|\mathrm{Re} (z)|\}$ for some $C>0$. If $f\mid_{(-1,1)}$ has an isolated zero at the origin, then $f$ vanishes to finite order at $0$.  
\end{theorem} 

It is known that the function 
$$
f(z)=\exp(-e^{i\pi/4}/\sqrt{z}) 
$$
is holomorphic on $\Delta^+$, extends $\mathcal{C}^\infty$-smoothly to $\overline{\Delta^+}$, and vanishes to infinite order at $0$ (see \cite{La91}). Hence, this example shows that the condition that $f$ maps $(-1,1)$ to a cone $\Gamma_C$ is essential.

M. Baouendi and L. Rothschild \cite{BR93} obtained the following result in which the condition $f\mid_{(-1,1)}$ has an isolated zero at $0$ is not neccessary.  
\begin{theorem}[\cite{BR93}]\label{T2}
Let $f$ be a holomorphic function defined on the upper half-disk $\Delta^+$ that  is $\mathcal{C}^0$  up to $(-1,1)\subset \partial\Delta^+$. Assume that $\mathrm{Re}f(x)\geq 0$ for every $x=\mathrm{Re} (z)\in (-1,1)$. Then $f$ has the boundary unique continuation property in the sense that if $f$ vanishes to infinite order at $0$, then $f = 0$.

\end{theorem} 

 In \cite{HKMP95}, Huang et al. posted the following conjecture.
 \begin{conjecture}[Huang et al.'s conjecture]
 Let $\Delta^+$ be the upper half disc in $\mathbb C$. Assume that $f$ is a holomorphic function on $\Delta^+$ with continuous extension up to $(-1,1)$ such that $f$ maps $(-1,1)$ into $\Gamma_C$, for some positive $C$. If $f$ vanishes to  infinite order at $0$, then $f$ vanishes identically.
 \end{conjecture}
Notice that if $C=1$, then $\mathrm{Re}[f^2(x)]\geq 0$ for every $x\in (-1,1)$. Therefore, it follows from Theorem \ref{T2} that this conjecture is true in case $C\leq 1$. 

In this paper, we prove the following theorem which ensures that Huang et al.'s conjecture is true for the case that $f$ is $\mathcal{C}^\infty$-smooth up to the boundary.
\begin{theorem} \label{Th2}
Suppose that $f\in \mathrm{Hol}(\Delta^+)\cap \mathcal{C}^\infty(\overline{\Delta^+})$ and $f(-1,1)\subset \Gamma_\infty:=\big(\mathbb C\setminus i\mathbb R\big)\cup \{0\}$. Suppose also that the set of zeros of $f\mid_{(-1,1)}$ is discrete and its limit point is $0$. If $f$ vanishes to infinite order at $0$, then $f\equiv 0$.
\end{theorem}

\begin{remark}\label{remark1}
The theorem is still true if $\Gamma$  is replaced by any domain $(\mathbb C\setminus L)\cup \{0\}$, where $L$ is a line in the complex plane through the origin.
\end{remark}

Throughout this paper, we assume that $f$ is holomorphic in $\Delta^+$, $\mathcal{C}^\infty$-smooth up to $(-1,1)$. The set of zeros of $f$ on $(-1,1)$ is discrete and its limit point is $0$. Each zero of $f$ on $(-1,1)\setminus\{0\}$ is of finite order. Let $\{r_n\}\subset \mathbb R^+$ be a decresing sequence such that all zeros of $f$ in $\overline{\Delta^+}$ lie on $\cup_n \gamma_{r_n}$, where $\gamma_r:=\{z\in \mathbb C \colon |z|=r,~\mathrm{Im}(z)\geq 0\}$ is an upper semi-circle with radius $r>0$. Denote $\kappa(n)$ by the number of zeros of $f$ on $\gamma_{r_n}\cap \Delta^+$, counted multiplicities, $\tilde\kappa(n)$ by the number of zeros of $f$ on $\gamma_{r_n}\cap (-1,1)$, counted multiplicities, and ${\kappa}{'}(n)$ by the number of zeros of $f$ on $\gamma_{r_n}\cap (-1,1)$ without counting multiplicities. In addition, denote $A_n:=\{re^{i\theta}\colon r_{n+1}<r<r_n,~0\leq \theta\leq \pi\}$. 

Let us recall that the index of a piecewise smooth curve $\gamma:[a,b]\to \mathbb C^*$ with respect to $0$ is the real number 
 $$
 \mathrm{Ind}(\gamma):=\mathrm{Re}\frac{1}{2\pi i}\int_\gamma \frac{dz}{z}=\mathrm{Re}\frac{1}{2\pi i}\int_a^b \frac{\gamma'(t)}{\gamma(t)}dt.  
 $$
 Notice that $\mathrm{Ind}(\gamma)<1/2$ if $\gamma\subset\Gamma_\infty\setminus \{0\}$ and $\mathrm{Ind}(\gamma)<1$ if $\gamma\subset \mathbb C^*\setminus \{iy\colon y>0\}$. For more properties of indices we refer the reader to \cite{La91}.
\begin{remark}\label{remark2} Let $f$ be a holomorphic function as in Theorem \ref{Th2}. Suppose that $f$ vanishes to infinite order at $b\in (-1,1)\setminus \{0\}$ with maximum modulus among zeros of $f$ on $(-1,1)$. Noting that $b$ is an isolated zero of $f$. Thus, by \cite[Lemma 2]{La91} (or Lemma \ref{le3} in Section \ref{s2}), there exists a sequence of upper semi-circles $\{\gamma_n\}$ with centre $b$ and radii $\epsilon_n$ such that $\epsilon_n \to 0^+$ and $\mathrm{Ind}(f\circ \gamma_n)\to +\infty$ as $n\to \infty$. Moreover, one can choose a sequence of upper semi-circles $\{\gamma'_n\}$ with centre $-b$ and radii $\epsilon'_n$ such that $\epsilon'_n\to 0^+$ and $\{\mathrm{Ind}(f\circ \gamma'_n)\}$ is bounded from blow. Now, fix $r,r>0$ with $r'<|b|<r$ and $|r-r'|$ small enough. We can construct a closed path $\gamma=\gamma_r+(-\gamma_{r'})+(-\gamma_n)+(-\gamma_n')+\sum_{j=1}^4 l_{n,j}$, where $l_{n,1}=[-r,-b-\epsilon'_n],l_{n,2}=[-b+\epsilon'_n,-r'],l_{n,3}=[r',b-\epsilon_n]$, and $ l_{n,4}=[b+\epsilon_n,r]$. Then, applying Cauchy's theorem to $f'/f$, one obtains that
 \begin{equation*}
 \begin{split}
 0&=\mathrm{Re}\frac{1}{2\pi i}\int\limits_\gamma \frac{f'(z)}{f(z)} dz\\
 & =\mathrm{Ind}(f\circ \gamma_r)-\mathrm{Ind}(f\circ \gamma_{r'})+\sum_{j=1}^4 \mathrm{Ind}(f\circ l_j) -\mathrm{Ind}(f\circ \gamma_n)- \mathrm{Ind} (f\circ \gamma_n').
 \end{split}
 \end{equation*}
 By the above argument, one has $\mathrm{Ind}(f\circ \gamma_n)+\mathrm{Ind}(f\circ \gamma'_n)\to +\infty$. Moreover, since $f\circ l_{n,j}\subset \Gamma_\infty$, for $1\leq j\leq 4$, one has
$$
|\sum_{j=1}^4 \mathrm{Ind}(f\circ l_{n,j})|<2.
$$
This is a contradiction since $|\mathrm{Ind}(f\circ \gamma_r)-\mathrm{Ind}(f\circ \gamma_{r'})|$ is finite.

Therefore, without loss of generality, we may assume that $f$ only vanishes to infinite order at $0$. 
\end{remark}
\begin{remark}
In Theorem \ref{Th2}, the condition that $f\in\mathcal{C}^\infty(\overline{\Delta^+})$ is the only technical condition but is important for showing the existence of left and right hand limits $I(r_n^\pm)$ (see the notation $I(r)$ below).   
\end{remark}
\begin{remark}[Notations]
Throughout the paper, taking the risk of confusion we employ the following notations
\begin{enumerate}
\item[(i)] $\Gamma_C:=\{z\in \mathbb C\colon |\mathrm{Im} (z)| \leq C|\mathrm{Re} (z)|\}$, $C>0$;
\item[(ii)] $\Gamma_\infty:=\{z\in \mathbb C\colon \mathrm{Re} (z)\ne 0\}\cup\{0\}$;
\item[(iii)]$\Delta^+:=\{z\in \mathbb C\colon |z|<1,~\mathrm{Im}(z)> 0\}$;
  \item[(iii)] $\mathbb R^+:=\{x\in \mathbb R\colon x>0\}$;
  \item[(iv)] $I(r):=\mathrm{Ind}(f\circ \gamma_{r})~(r>0)$, where $\gamma_r:=\{z\in \mathbb C\colon |z|=r,~\mathrm{Im}(z)\geq 0\}$.
\end{enumerate}
Let $f,g : A \to \mathbb C$ be functions defined on $A\subset \mathbb C$ with $0\in \overline{A}$ such that $\lim_{z\to 0} f(z)=\lim_{z\to 0} g(z)= 0$. We write 
\begin{enumerate}
\item[(v)] $f\sim g$ at $0$ on $A$ if $\lim_{z\to 0}f(z)/g(z) = 1$;
\item[(vi)]$f \approx g$ at $0$ on $A$ if there is $C>0 $ such that $1/C |g(z)|\leq |f(z)|\leq C |g(z)|$ for all $z \in A$.
\end{enumerate}
  \end{remark}
 This paper is organized as follows. In Section \ref{s2}, we give several lemmas needed later and then we prove Theorem \ref{Th2}.  Some regularity properties of the Riemann mapping functions on the boundary and an example are given in Section \ref{s3}. 
 \section{Finite order vanishing of boundary values of holomorphic functions} \label{s2}
 In this section, we shall prove some technical lemmas which generalize Lakner's lemmas (cf. \cite{La91}). The proofs follow very closely the proofs of \cite{La91}. In fact, M. Lakner proved his lemmas for the case that $f$ has no zeros in $(-1,1)$, except the origin. In this paper, we consider the case that the origin is not an isolated zero of the restricted function $f\mid_{(-1,1)}$. Therefore, in order to prove these lemmas, we shall modify Lakner's method for the general case.
 
  First of all, we have the following lemma.
 \begin{lemma}\label{le1}Suppose that $f\in \mathrm{Hol}(\Delta^+)\cap \mathcal{C}^\infty(\overline{\Delta^+})$ and $f(-1,1)\subset \Gamma:=\mathbb C\setminus i\mathbb R^+$. Suppose also that the set of zeros of $f\mid_{(-1,1)}$ is discrete, its limit point is $0$, and each zero of $f$ on $(-1,1)\setminus \{0\}$ is of finite order. Then, we have
\begin{enumerate} 
\item[(i)] $I(r_n^+)-I(r_n^-)=\kappa(n)+\frac{\tilde \kappa(n)}{2}$;
\item[(ii)] $|I(r)-I(r')|<2,~r_{n+1}< r,r'< r_n$.
\end{enumerate}
\end{lemma}
\begin{proof}

\noindent
{\bf(i)}
For each $n$ we can write $f$ in the form 
$$
f(z)=(z-\alpha_1)^{l_1}\cdots (z-\alpha_m)^{l_m} \varphi(z),
$$
where $\alpha_1,\ldots,\alpha_{m}$ are all zeros of $f$ on $\gamma_{r_n}$ and $\varphi$ is a continuous function without zeros on ${A_{n-1}\cup A_n}\cup\gamma_{r_n}$, holomorphic in the interior. 

Therefore, we have
$$
\frac{f'(z)}{f(z)}=\sum_{j=1}^m \frac{l_j}{z-\alpha_j}+\frac{\varphi'(z)}{\varphi(z)},
$$
and thus 
$$
I(r)=\sum_{j=1}^m l_j \mathrm{Ind}(\gamma_r-\alpha_j)+\mathrm{Ind}(\varphi\circ \gamma_r). 
$$

At the point $a\in \Delta^+$ with $|a|=r_n$, $\lim_{r\to r_n^+}\mathrm{Ind}(\gamma_r-a)=3/4$ and $\lim_{r\to r_n^-}\mathrm{Ind}(\gamma_r-a)=-1/4$. Moreover, at the point $b\in \gamma_{r_n}\cap (-1,1)$, we have 
$\lim_{r\to r_n^+}\mathrm{Ind}(\gamma_r-b)=1/2$ and $\lim_{r\to r_n^-}\mathrm{Ind}(\gamma_r-b)=0$.  

Therefore, we have
\begin{equation*} \label{eq4}
\begin{split}
I(r_n^+)&=\lim_{r\to r_n^+}\mathrm{Ind}I(r)=3\kappa(n)/4+\tilde{\kappa}(n)/2 +\mathrm{Ind}(\varphi\circ \gamma_{r_n}),\\
I(r_n^-)&=\lim_{r\to r_n^-}\mathrm{Ind} I(r)=-\kappa(n)/4+\mathrm{Ind}(\varphi\circ \gamma_{r_n}).  
\end{split}
\end{equation*}   
Hence, we conclude that $I(r_n^+)-I(r_n^-)=\kappa(n)+\tilde \kappa(n)/2$.

\noindent
{\bf(ii)} Fix $r_{n+1}<r'<r<r_n$ and construct a closed path $\gamma=\gamma_{r}+(-\gamma_{r'})+l^++l^-$, where $l^+=[r',r]$ and $l^-=[-r,-r']$. Employing Cauchy's theorem to $f'/f$ we obtain
$$ 
0 =\mathrm{Re} \frac{1}{2\pi i} \int_\gamma \frac{f'}{f} dz=I(r)+\mathrm{Ind}(f\circ l^+)-I(r')+\mathrm{Ind}(f\circ l^-).
$$
Since the paths $f\circ l^\pm\subset \Gamma$, $|\mathrm{Ind}(f\circ l^\pm)|<1$. This implies that $|I(r)-I(r')|<2$.
\end{proof}

\begin{lemma} \label{le2} Suppose that $f\in \mathrm{Hol}(\Delta^+)\cap \mathcal{C}^\infty(\overline{\Delta^+})$ and $f(-1,1)\subset \Gamma:=\big(\mathbb C\setminus i\mathbb R^+\big) \cup \{0\}$. Suppose also that the set of zeros of $f\mid_{(-1,1)}$ is discrete, its limit point is $0$, and each zero of $f$ on $(-1,1)\setminus \{0\}$ is of finite order. Then, $I(r)$ is bounded from above if one of following conditions is satisfied:
\begin{enumerate}
\item[(i)] $\Gamma$ is the cone $(\mathbb C\setminus L)\cup \{0\}$, where $L$ is a line through the origin.
\item[(ii)] Each zero of $f$ on $(-1,1)$ is of order at least $2$. 
\item[(iii)] $\sum_{n=1}^\infty \kappa(n)+\frac{1}{2}\sum_{n=1}^\infty \tilde \kappa(n) -\sum_{n=1}^\infty \kappa'(n)=+\infty$. 
\item[(iv)] $\Gamma$ is a half-plane $\{z\in \mathbb C\colon \mathrm{Re}(az)\geq 0\}$ for some $a\in \mathbb C^*$.
\end{enumerate}
\end{lemma}
\begin{proof}
In case the origin is an isolated zero of $f$, one easily see that $I(r)$ is bounded. Therefore, we consider the case that the set of zeros of $f$ in $ \overline{\Delta^+}$ is a sequence of points converging to the origin. 

Let $\{r_n\}\subset \mathbb R^+$ be a decreasing sequence such that all zeros of $f$ in $\overline{\Delta^+}$ lie on $\cup \gamma_{r_n}$. Now, for $\epsilon>0$ small enough we consider a closed path $\gamma_n^\epsilon=\gamma_{r_n-\epsilon}+(-\gamma_{r_{n+1}+\epsilon})+l_{n,\epsilon}^+ +l_{n,\epsilon}^-$, where $l_{n,\epsilon}^+=[r_{n+1} +\epsilon,r_n-\epsilon]$ and $l_{n,\epsilon}^-=[-r_{n}+\epsilon,-r_{n+1}-\epsilon]$. Then, employing Cauchy's theorem to $f'/f$ we obtain
$$ 
0 =\mathrm{Re} \frac{1}{2\pi i} \int_{\gamma_n^\epsilon} \frac{f'}{f} dz=-I(r_{n+1}+\epsilon)+\mathrm{Ind}(f\circ l_{n,\epsilon}^+)+I(r_n-\epsilon)+\mathrm{Ind}(f\circ l_{n,\epsilon}^-).
$$
Letting $\epsilon\to 0^+$, one obtains that

$$ 
 I(r_n^-) -I(r_{n+1}^+)+s(n)=0,
$$
where $s(n):=\lim_{\epsilon\to 0^+}\Big[\mathrm{Ind}(f\circ l_{n,\epsilon}^+)+\mathrm{Ind}(f\circ l_{n,\epsilon}^-)\Big]$.

From this and Lemma \ref{le1} (i), we have
\begin{equation}\label{eq5}
\begin{split}
\kappa(n)+\tilde \kappa(n)/2=I(r_n^+)-I(r_{n+1}^+) +s(n).
\end{split}
\end{equation}
Summing the first $N$ relations (\ref{eq5}) we obtain
\begin{equation}\label{eq6}
\begin{split}
I(r_1^+)-I(r_{N+1}^+)= \sum_{n=1}^N \kappa(n)+\frac{1}{2}\sum_{n=1}^N\tilde \kappa(n)-\sum_{n=1}^N s(n).
\end{split}
\end{equation}

Now, fix $N\geq 1$. We shall prove that $\sum_{n=1}^N s(n)\leq\sum_{n=1}^N\kappa'(n)$ for case $f(-1,1)\subset \mathbb C\setminus i\mathbb R^+ $ and $\sum_{n=1}^N s(n)\leq 1/2\sum_{n=1}^N\kappa'(n)$ for case $f(-1,1)\subset (\mathbb C\setminus L)\cup \{0\}$, where $L$ is a line through the origin. Indeed, suppose that there exist $i,j\in \mathbb N^*$ with $i<j$ such that $f(r_i)=f(r_j)=0$ and $f(r_k)\ne 0$ for every $i<k<j$. Then, if $f(-1,1)\subset \mathbb C\setminus i\mathbb R^+$, then 
$$
\lim_{\epsilon\to 0^+}\sum_{k=i}^{j-1} \mathrm{Ind}(f\circ l_{k,\epsilon}^+)= \lim_{\epsilon\to 0^+} \mathrm{Ind}(f\circ [r_j+\epsilon,r_i-\epsilon])\leq 1
$$
since $f\circ [r_j,r_i]\subset \mathbb C\setminus i\mathbb R^+$. Similarly, if $f(-1,1)\subset (\mathbb C\setminus L)\cup \{0\}$, then  
$$
\lim_{\epsilon\to 0^+}\sum_{k=i}^{j-1} \mathrm{Ind}(f\circ l_{k,\epsilon}^+)= \lim_{\epsilon\to 0^+} \mathrm{Ind}(f\circ [r_j+\epsilon,r_i-\epsilon])\leq \frac{1}{2}
$$
since $f\circ [r_j,r_i]$ is contained in a half-plane. Furthermore, one can consider the sequence of points $\{-r_j\}$ instead of $\{r_j\}$ and one thus obtains similar estimates. Hence, these estimates yield our assertions.   

Therefore, if $\sum_{n=1}^\infty \kappa(n)+\frac{1}{2}\sum_{n=1}^\infty \tilde \kappa(n) -\sum_{n=1}^\infty \kappa'(n)= +\infty$, then $\sum_{n=1}^\infty \kappa(n)+\frac{1}{2}\sum_{n=1}^\infty \tilde \kappa(n) -\sum_{n=1}^\infty s(n)=+\infty$, and thus by (\ref{eq6}) one has $I(r_n^+) \to-\infty$. Consequently, by Lemma \ref{le1} $\mathrm{(ii)}$, it follows that $I(r)$ is bounded from above. This proves the assertion for the case $\mathrm{(iii)}$.

For the case $\mathrm{(ii)}$, since $\sum_{n=1}^N\tilde \kappa(n)\geq 2\sum_{n=1}^N\kappa'(n)$, it follows from (\ref{eq6}) that either $\{I(r_n^+)\}$ is bounded or $I(r_n^+)\to -\infty$ as $n\to \infty$. So by Lemma \ref{le2}, $I(r)$ is bounded from above.

Next, if  $\Gamma$ is an infinite  cone $(\mathbb C\setminus L)\cup \{0\}$, where $L$ is a line through the origin, then 
$$
\sum_{n=1}^N s(n)\leq \frac{1}{2}\sum_{n=1}^N\kappa'(n)\leq \frac{1}{2}\sum_{n=1}^N\tilde\kappa(n).
$$
Therefore, by (\ref{eq6}) and Lemma \ref{le2} we also have $I(r)$ bounded from above, and hence, $\mathrm{(i)}$ is proved.

Finally, if $\Gamma$ is a half-plane, then $\sum_{n=1}^N s(n)\leq 1/2\sum_{n=1}^N\kappa'(n)$ for any $N\geq 1$. Hence, (\ref{eq6}) implies that $I(r)$ is bounded from above. Thus, this proves $\mathrm{(iv)}$.
\end{proof}

The following lemma is a generalization of Lemma 1 in \cite{La91}.
\begin{lemma}\label{le3}
Suppose that $f\in \mathrm{Hol}(\Delta^+)\cap \mathcal{C}^\infty(\overline{\Delta^+})$ and $f(-1,1)\subset \Gamma:=\mathbb C\setminus i\mathbb R^+$. Suppose also that the set of zeros of $f\mid_{(-1,1)}$ is discrete and its limit point is $0$. If $f$ only vanishes to infinite order at $0$, then 
$$
\limsup_{r\to 0^+} \frac{1}{\ln(1/r)}\int_r^1 \frac{I(t)}{t} dt=+\infty.
$$ 
\end{lemma}
\begin{proof}
Without loss of generality, we may assume that there exists a sequence of zeros of $f$ converging to the origin.

Let $\{r_n\}\subset \mathbb R^+$ be a sequence with $r_n \to  0^+$, such that all zeros of $f$ lie on $\cup \gamma_{r_n}$.  Denote $A_n:=\{re^{i t}\colon 0\leq t\leq \pi,~r_{n+1}<r<r_{n}\}$. Then, on each $A_n$ there exists a holomorphic function $\Phi(z):=u_n(z)+iv_n(z)$ such that $f(z)=e^{\Phi(z)}$. One can see that $u_n(z)=\ln |f(z)|$ on $A_n$. Hence, we have
\begin{equation} \label{eq1}
\begin{split}
I(r)&:= \mathrm{Ind}(f\circ \gamma_r)\\
&=\mathrm{Re} \frac{1}{2\pi i}\int_{\gamma_r} \frac{f'(z)}{f(z)} dz=\mathrm{Re} \frac{1}{2\pi i}\int_{\gamma_r} \Phi'(z) dz\\
&=\mathrm{Re} \frac{1}{2\pi i} \big(\Phi(re^{i\pi})-\Phi(r)\big)= \frac{1}{2\pi } \big(v(re^{i\pi})-v(r)\big)\\
&=\frac{1}{2\pi } \int_0^\pi \frac{d}{dt} v(r,t) dt= \frac{1}{2\pi } \int_0^\pi  v_\theta(r,\theta) d\theta \\
&= \frac{1}{2\pi } \int_0^\pi  r u_r(r,\theta) d\theta.
\end{split}
\end{equation}

By Lemma \ref{le1}, $I(r)$ is piecewise continuous on $(0,1]$ and thus it is integrable on $[r,1]$. Therefore, let us denote 
$$
J(r):=\frac{1}{\ln 1/r} \int_r^1 \frac{I(t)}{t} dt.
$$  
Then, by (\ref{eq1}) we have
\begin{equation} \label{eq2}
\begin{split}
J(r)= \frac{1}{\ln 1/r} &\int_r^1 \frac{I(t)}{t} dt\\
= \lim_{\epsilon\to 0^+}\frac{1}{\ln 1/r}&\Big [ \sum_{r<r_n<1}\int_{r_{n}e^{\epsilon}}^{r_{n-1}e^{-\epsilon}} \frac{I(t)}{t}dt+\int_r^{r_{n_0}e^{-\epsilon}} \frac{I(t)}{t}dt  \Big] \\
= \lim_{\epsilon\to 0^+}\frac{1}{\ln 1/r}&\Big [ \sum_{r<r_n<1}\int_{r_{n}e^{\epsilon}}^{r_{n-1}e^{-\epsilon}}  \frac{1}{t} \int_0^\pi t \frac{\partial u_n}{\partial t}(t,\theta) d\theta dt+\int_r^{r_{n_0}e^{-\epsilon}} \frac{1}{t} \int_0^\pi t \frac{\partial u_n}{\partial t}(t,\theta) d\theta dt  \Big] \\
= \lim_{\epsilon\to 0^+}\frac{1}{\ln 1/r}&\Big [ \sum_{r<r_n<1}\int_0^\pi\int_{r_{n}e^{\epsilon}}^{r_{n-1}e^{-\epsilon}}  \frac{\partial u_n}{\partial t}(t,\theta)  dtd\theta\\
&+\int_0^\pi\int_r^{r_{n_0}e^{-\epsilon}}  \frac{\partial u_n}{\partial t}(t,\theta) dtd\theta   \Big]\\
= \lim_{\epsilon\to 0^+}\frac{1}{\ln 1/r}&\Big [ \sum_{r<r_n<1}\int_0^\pi \big( u_n(r_{n-1}e^{-\epsilon},\theta)-u_n(r_{n}e^{\epsilon},\theta)\big) d\theta\\
&+\int_0^\pi \big( u_n(r_{n_0}e^{-\epsilon},\theta)-u_n(r,\theta)\big)d\theta   \Big],
\end{split}
\end{equation}
where $n_0$ is the integer satisfying that $r_{n_0+1}<r<r_{n_0}$.

We will show that $\lim_{\epsilon\to 0^+}\Big(u_n(r_ne^{\epsilon},\theta)-u_n(r_ne^{-\epsilon},\theta)\Big)=0,~0\leq \theta\leq \pi $. Indeed, for each $n$ we can write $f$ in the form 
$$
f(z)=(z-\alpha_1)^{l_1}\cdots (z-\alpha_m)^{l_m} \varphi(z),
$$
where $\alpha_1,\ldots,\alpha_{m}$ are all zeros of $f$ on $\gamma_{r_n}$ and $\varphi$ is a continuous function without zeros on $A_{n-1}\cup A_n \cup \gamma_{r_n}$, holomorphic in the interior. Thus, we obtain that
\begin{equation*} \label{eqq2}
\begin{split}
u_n(r_ne^\epsilon,\theta)-u_{n}(r_ne^{-\epsilon},\theta)&= \sum_{j=1}^m l_j \ln \frac{|(r_ne^\epsilon)e^{i\theta}-\alpha_j|}{|(r_ne^{-\epsilon})e^{i\theta}-\alpha_j|}+\ln \frac{|\varphi((r_n+\epsilon)e^{i\theta})|}{|\varphi((r_n-\epsilon)e^{i\theta})|} \\
&=\sum_{j=1}^m l_j \epsilon+ \ln \frac{|\varphi((r_n+\epsilon)e^{i\theta})|}{|\varphi((r_n-\epsilon)e^{i\theta})|}\to 0
\end{split}
\end{equation*}
as $\epsilon\to 0$, since $\frac{|(r_ne^\epsilon)e^{i\theta}-\alpha_j|}{|(r_ne^{-\epsilon})e^{i\theta}-\alpha_j|}=e^{\epsilon}$. Hence, $\lim_{\epsilon\to 0^+}\big(u(r_ne^{\epsilon},\theta)-u(r_ne^{-\epsilon},\theta)\big)=0$.

Thanks to the fact that $\lim_{\epsilon\to 0^+}\big(u(r_ne^{\epsilon},\theta)-u(r_ne^{-\epsilon},\theta)\big)=0$ for every $n\in \mathbb N$ and by the mean value theorem, (\ref{eq2}) becomes
\begin{equation*} \label{eq3}
\begin{split}
J(r)&=  \frac{1}{\ln 1/r}\int_0^\pi \big( u(1,\theta)-u(r,\theta)\big) d\theta\\
&=\frac{O(1)}{\ln 1/r} -\frac{u(r,\theta_r)}{\ln 1/r},
\end{split}
\end{equation*}
for some $\theta_r\in [0,\pi]$.

Now we shall prove $\limsup_{r\to 0^+} J(r)=+\infty$. Indeed, suppose otherwise that $J(r)$ is bounded from above, i.e. there exists $C>0$ such that 
$$
\frac{-u(r,\theta_r)}{\ln 1/r}\leq C. 
$$
Then, it follows that
$$
\frac{1}{|f(r e^{i\theta_r})|}=e^{-u(r,\theta_r)}\leq  e^{C\ln 1/r}. 
$$
Therefore, $|f(r e^{i\theta_r})|\geq r^C$, which is a contradiction since $f$ vanishes to infinite order at $0$.
\end{proof}

\begin{proof}[Proof of Theorem \ref{Th2}]
Suppose that there exists a non-zero function $f$ satisfying conditions as in Theorem \ref{Th2}. By Remark \ref{remark2}, we may assume that $f$ only vanishes to infinite order at $0$. Therefore, by Lemma \ref{le2} one has that $I(r)$ is bounded from above. This implies that $J(r)$ is also bounded from above, which contradicts Lemma \ref{le3}. Hence, the proof is complete.
\end{proof}

By Theorem \ref{Th2}, we have the following corollary, which gives a confirmative answer to Huang et al.'s conjecture for the case that $f$ is $\mathcal{C}^\infty$-smooth up to the boundary.
\begin{corollary} Suppose that $f$ is a holomorphic function on $\Delta^+$ with $\mathcal{C}^\infty$-smooth extension up to $(-1,1)$ such that $f$ maps $(-1,1)$ into $\Gamma_C$, for some positive $C$. Suppose also that the set of zeros of $f\mid_{(-1,1)}$ is discrete and its limit point is $0$. If $f$ vanishes to  infinite order at $0$, then $f$ vanishes identically.
\end{corollary}

However, at present we do not know whether Theorem \ref{Th2} holds for the case $f(-1,1)\subset \mathbb C\setminus i\mathbb R^+$.
\section{Asymptotic behaviour of the Riemann mapping function at a cusp}\label{s3}
In this section, we recall some results that pertain to the behaviour of the Riemann mapping function at a cusp and prove a result relating to the boundary uniqueness of holomorphic functions. Moreover, an example concerning Huang et al.'s conjecture is also given.

First of all, we now recall the general H\"{o}lder continuity (see \cite{Khanh13}). 
\begin{define}[\cite{Khanh13}]
Let $F$ be an increasing function such that $\lim_{t\to +\infty}F(t)=+\infty$. For $\Omega\subset \mathbb C^n$, define $F$-H\"{o}lder space on $\bar \Omega$ by
$$
\Lambda^F(\overline{\Omega})=\{u\colon \|u\|_\infty+\sup_{z,z+h\in \bar\Omega}F(|h|^{-1})|u(z+h)-u(z)|<+\infty\}.
$$
\end{define}
Notice that the $F$-H\"{o}lder space includes the standard H\"{o}lder space $\mathcal{C}^{0,\alpha}(\overline{\Omega})$ by taking $F(t)=t^\alpha$ with $0<\alpha<1$.

\begin{define} Let $f,g: \Omega \to \mathbb C$ be two infinitesimal functions as $x \to a$ satisfying that $g(x)\ne 0$ for any $x\in \Omega$. We say that $f$ is infinitesimal with respect to $g$ as $x\to a$ if 
$$    
\lim_{\Omega\ni x\to a}\frac{|f(x)|}{|g(x)|^N} = 0, 
$$
for every $N\in \mathbb N$. 
\end{define}

The well-known Riemann Mapping Theorem states that there exists a conformal map of a simply connected proper domain of the complex plane onto the unit disc. This mapping is known as a Riemann mapping function. The rest of this section is devoted to a study of the extendability of a Riemann mapping function to the boundary.

 In the case of a simply connected bounded domain with smooth boundary, the Riemann mapping function extends smoothly to the boundary of the domain. For the case of singular boundary point, L. Lichtenstein and S. Warschawski \cite{Wa42, Wa55, Li11} investigated the asymptotic behaviour of the mapping function at an analytic corner, where the opening angle of two regular analytic arcs is greater than $0$. The mapping function behaves like $z^{1/\alpha}$, where $\pi \alpha$ is the opening angle of the analytic corner. For more details we refer the reader to Pommerenke \cite[Chapter 3]{Po91}. When the opening angle vanishes, we recall the definition of analytic cusps. 
 \begin{define}[\cite{KL16, Pr15}] One says that $\Omega$ has an analytic cusp at $0$ if the boundary of $\Omega$ at $0$ has two regular analytic curves such that the openning angle of $\Omega$ at $0$ vanishes.
\end{define}

It is shown in \cite{KL16} that, after a change of variable if necessary, one can consider a domain $\Omega$ with an analytic cusp at $0$ given by $\Omega=\{0<x<1, 0<y<\alpha(x)\}$, where $\alpha:(-1,1)\to \mathbb R$ is a real-analytic curve. Let $\Phi: \Omega\to \Delta^+$ be a Riemann mapping function which maps $0$ to $0$.

In the case $\partial\Omega$ has an analytic cusp at $0$ satisfying that $\alpha(t)=\sum_{j=N}^\infty a_j t^j$, T. Kaiser and S. Lehner \cite{KL16} achieved an asymptotic behaviour of the Riemann mapping function
$$
\Phi(z)\sim z^\sigma \exp \Big(\frac{c_0}{z^N}+\frac{c_1}{z^{N-1}}+\cdots+\frac{c_{N-1}}{z}\Big),
$$
for some $\sigma\in \mathbb R$. In general,  S. Warschwaski \cite[Theorem XI(A) and Theorem XI(B)]{Wa42} proved that the Riemann mapping function $\Phi$ satisfies 
$$
|\Phi(z)|\sim  \exp \Big(-\pi \int_t^a \frac{dr}{r \alpha(r)}\Big)=: F(t).
$$
In other words, $\Phi$ is $F$-H\"{o}lder on $\overline{\Omega}$.
\begin{lemma} Let $D$ be a domain in $\mathbb C$ with $0\in \partial D$ and let $\Phi:D \to \Delta^+$ be a Riemann mapping function such that $\Phi(0)=0$. If $f: D \to \mathbb C$ is infinitesimal with respect to $\Phi$, then $f\circ \Phi^{-1}$ vanishes to infinite order at $z=0$.
\end{lemma}
\begin{proof}
Since $f: D \to \mathbb C$ is infinitesimal with respect to $\Phi$, it follows, for any $n\in\mathbb N$, that 
\begin{equation*}
\begin{split}
\lim_{z\to 0}\frac{f\circ \Phi^{-1}(z)}{|z|^n}=\lim_{z\to 0}\frac{f\circ \Phi^{-1}(z)}{|\Phi\circ \Phi^{-1} (z)|^n}=\lim_{w\to 0}\frac{f(w)}{|\Phi(w)|^n}=0.
\end{split}
\end{equation*}
Therefore, $f\circ \Phi^{-1}$ vanishes to infinite order at $z=0$, and hence the proof is complete.
\end{proof}
We recall the following corolary which is given in \cite[Claim 2.3]{DK15}.
\begin{corollary}[See \cite{DK15}]\label{Cor2} Let $D$ be a domain in $\mathbb C$ with $0\in \partial D$ and let $\Phi:D \to \Delta^+$ be a Riemann mapping function such that $\Phi(0)=0$. If $\partial D \in \mathcal{C}^{0,\alpha}$, then $\Phi^{-1}\in \mathcal{C}^{0,\alpha} (\overline{\Delta^+})$ and $f\circ \Phi^{-1}$ vanishes to infinite order at $0$ for any function $f:D \to \mathbb C$ vanishing to infinite order at $0$. 
\end{corollary}

\begin{example}
Let $\Omega$ be the domain defined by
$$
\Omega:=\{z\in \mathbb C\colon 0<x<1,~0<y<x^2\}.
$$

By \cite[Theorem XI(A) and Theorem XI(B)]{Wa42} (see also the proof of Proposition $2.2$ in \cite{KL16}) and a computaion, we obtain
$$
|\Phi(z)|\sim  \exp \Big(-\pi \int_t^a \frac{dr}{r^3}\Big)=\exp \Big(\frac{\pi}{2}\big(\frac{1}{a^2}-\frac{1}{t^2}\big) \Big)\approx \exp(-\frac{\pi}{2t^2}).
$$
Then, the function defined on $\mathbb C$ by
\begin{equation*}
f(z)=
\begin{cases}
\exp(-C/|z|^\alpha)&~\text{if}~ z\ne 0\\
0&~\text{if}~ z= 0
\end{cases},
\end{equation*}
where $C>\pi/2$ for $\alpha=2$ and $C>0$ for $\alpha >2$, is infinitesimal with respect to $\Phi$. Moreover, $f\circ \Phi^{-1}$ vanishes to infinite order at $0$.
\end{example}

We now prove the main result of this section.
\begin{proposition}\label{pro5} Let $D\subset \mathbb C$ be a bounded domain in $\mathbb C$ with $0\in \partial D$ and let $\psi: \Delta^+\to D $ be a Riemann mapping function that is $\mathcal{C}^0$ up to $(-1,1)\in \partial \Delta^+$ satisfying $\psi(0)=0$. Assume that $f$ be a holomorphic function defined on $D$ that  is $\mathcal{C}^0$ up to $\partial D$ such that 
$$
\mathrm{Re}f(p)\ne 0,~\forall p\in D.
$$
If $f$ is infinitesimal with respect to $\psi^{-1}$ and vanishes to infinite order at $0$, then $f\equiv 0$. 
\end{proposition}
\begin{proof}
Without loss of generality, we may assume that 
$$
\mathrm{Re}(p)>0,~\forall p\in D.
$$
Since $f$ is infinitesimal with respect to $\psi^{-1}$, $(f \circ \psi) \in \mathrm{Hol}(\Delta^+
)\cap \mathcal C^{0}(\overline{\Delta^+})$
and $(f \circ \psi$) vanishes to infinite order at $0$. Moreover, from the hypothesis 
$$
\mathrm{Re}f(p)>0,~\forall p\in D,
$$
one has $\mathrm{Re}(f \circ \psi)(x)\geq 0$ with $x:=\mathrm{Re}(z)\in (-1,1)$. Therefore, by Theorem \ref{T2} we conclude $f\circ \psi\equiv 0$, and thus $f\equiv 0$.
\end{proof}

As an application, we obtain the following corollary which is a generalization of \cite[Lemma 2.4]{DK15}. 
\begin{corollary} \label{Cor3} Let $f$ be a holomorphic function defined on the upper half-disk $\Delta^+$ that  is $\mathcal{C}^0$  up to $(-1,1)\subset \partial\Delta^+$. Assume that there exists a simply connected domain $\hat{B}\subset \Delta^+$, with boundary $\partial \hat{B}\ni 0$, of class $\mathcal{C}^{0,\alpha}$, for some $\alpha>0$, such that 
$$
\mathrm{Re}f(p)\ne 0,~\forall p\in \hat{B}.
$$
If $f$ vanishes to infinite order at $0$, then $f\equiv 0$. 
\end{corollary}
\begin{proof}
Without loss of generality, we may assume that 
$$
\mathrm{Re}(p)>0,~\forall p\in \hat{B}.
$$
Then, by Corollary \ref{Cor2}, one can find a biholomorphism $\psi:\Delta^+\to \hat{B}$, which is $\mathcal{C}^{0,\alpha}$ up to the boundary, such that $(f \circ \psi) \in \mathrm{Hol}(\Delta^+
)\cap \mathcal C^{0,\alpha}(\overline{\Delta^+})$,
and such that $(f \circ \psi$) vanishes to infinite order at $0$. Therefore, Proposition \ref{pro5} tells us that $f\equiv 0$. 
\end{proof}

Now, we assume that $f\in \mathrm{Hol}(\Delta^+)\cap \mathcal{C}^\infty(\overline{\Delta^+})$ and $f$ vanishes to infinite order at $0$. Then, it follows from Corollary \ref{Cor3} that there exists a sequence $\{p_j\}\subset \Delta^+$ such that $p_j\to 0$ and $\mathrm{Re}~f(p_j)=0$. In addition,  A. Daghighi and S. Krantz \cite{DK16} recently proved that either $f\equiv 0$ or there is a sequence in $\Delta^+$, converging to $0$, along which $\mathrm{Im}(z)/\mathrm{Re}(z)$ is unbounded.

It is a natural question whether there exists a subdomain $\hat{B}$ with $0\in \partial\hat B$ such that $\mathrm{Re}(f(z))\ne 0$ for all $z\in \hat B$. However, the following example points out that one cannot remove the condition $f(-1,1)\subset \Gamma_\infty$ in Theorem \ref{Th2} and there is no such a subdomain $\hat B\subset \Delta^+$ with $0\in \partial\hat B$ such that $\mathrm{Re}(f(z))\ne 0$ on $\hat B$.

In order to introduce the example, we need the following well-known result. 
 \begin{theorem}[See Theorem $9.1.5$ in \cite{GK97}]\label{T5} If $\{a_j\}_{j=1}^\infty\subset \Delta$ (with possible repetitions) satisfies 
 $$
 \sum_{j=1}^\infty (1-|a_j|)<+\infty
 $$
and no $a_j=0$, then there is a bounded holomorphic fuction on $\Delta$ which has zero set consisting precisely of the $a_j$'s, counted according to their multiplicities. Specifically, the infinite product 
$$
\prod_{j=1}^\infty \frac{-a_j}{|a_j|} B_{a_j}(z)
$$ 
converges uniformly on compact subsets of $\Delta$ to a bounded holomorphic function $B(z)$, where $B_{a_j}(z)=\dfrac{z-a_j}{1-\bar a_j z}$. The zeros of $B$ are precisely the ${a_j}$'s, counted according to their multiplicities.
\end{theorem}
\begin{example} Denote $\mathbb H:=\{z\in \mathbb C\colon \mathrm{Im}(z)>0\}$ the upper half plane and consider the sequence $\{a_{m,n}\}_{m\in \mathbb Z, n\in \mathbb N^*}\subset \mathbb H$ given by
$$
a_{m,n}=\frac{1}{m^3-in^3}\in \mathbb H.
$$
Let $\psi: \mathbb H\to \Delta$ be the biholomorphism defined by 
$$
\psi(z)=\frac{z-i}{z+i}
$$  
and let us denote
 $$
 \alpha_{m,n}:=\psi(a_{m,n})= \frac{1-n^3-im^3}{1+n^3+im^3}=\frac{1-m^6-n^6-2im^3}{(1+n^3)^2+m^6}.
 $$
Then, a computation shows that 
\begin{equation*}
\begin{split}
1-|\alpha_{m,n}|^2&=\frac{(m^6+n^6+2n^3+1)^2-(m^6+n^6-1)^2-4m^6}{(m^6+n^6+2n^3+1)^2}\\
&=\frac{4(m^6+n^6+n^3)(n^3+1)-4m^6}{(m^6+n^6+2n^3+1)^2}\\
&=\frac{4n^3(m^6+n^6+n^3)+4(n^6+n^3)}{(m^6+n^6+2n^3+1)^2}\\
&<\frac{4(n^3+1)}{m^6+n^6+2n^3+1}\\
&<\frac{4}{n^3+1}.
\end{split}
\end{equation*}
 Hence, it yields that 
$$
\sum_{m,n} (1-|\alpha_{m,n}|)<+\infty.
$$ 

Now, we define the holomorphic function $G:\mathbb H\to \mathbb C$ by
$$
G(w)=
\prod_{n=1}^\infty \prod_{m=-\infty}^\infty\frac{-\alpha_{m,n}}{|\alpha_{m,n}|} B_{\alpha_{m,n}}(z).
$$ 
Then the function $F:=G\circ \psi: \mathbb H\to \Delta$ satisfies $|F(z)|\leq 1$ for every $z\in \mathbb H$. Therefore, the function 
$$
f:=e^{-\frac{e^{i\pi/4}}{\sqrt{z}}}F(z)
$$ 
is holomorphic in $\mathbb H$, continuous in $\overline{\mathbb H}$ and it vanishes to infinite order at $z=0$. Moreover, the origin is only the isolated zero of $f\mid_{\partial\mathbb H}$.
\end{example}
\begin{Acknowlegement} This work was completed when the first author was visiting the Vietnam Institute for Advanced Study in Mathematics (VIASM). He would like to thank the VIASM for financial support and hospitality.  
\end{Acknowlegement}

\end{document}